\documentclass[12pt, twoside, reqno] {amsart}

\usepackage{amsfonts}
\usepackage{amssymb}
\usepackage{latexsym}
\usepackage{graphicx,graphics}
\usepackage{epsf,graphicx,latexsym%
}

\newcommand{\be} {\begin{eqnarray}}
\newcommand{\ee} {\end{eqnarray}}
\newcommand{\bep} {\begin{eqnarray*}}
\newcommand{\eep} {\end{eqnarray*}}

\newcommand {\s}{\mathop{\widehat{S}}\nolimits}

\newcommand {\Hol}{\mathop{\rm Hol}\nolimits}

\newcommand {\Id}{\mathop{\rm Id}\nolimits}

\renewcommand {\Re}{\mathop{\rm Re}\nolimits}

\newcommand{\N}{{\mathbb N}}
\newcommand{\B}{{\mathbb B}}
\newcommand{\C}{{\mathbb C}}

\newcommand {\D}{\mathbb{D}}

\newtheorem{remar}{Remark}[section]
\newtheorem{examp}{Example}[section]
\newtheorem{defin}{Definition}[section]
\newtheorem{corol}{Corollary}[section]
\newtheorem{propo}{Proposition}[section]
\newtheorem{theorem}{Theorem}[section]
\newtheorem{lemma}{Lemma}[section]
\newtheorem{remark}{Remark}[section]

\newtheorem{assump}{Assumption}

\newcommand{\rema}{\begin{remar}\rm}
\newcommand{\erema}{$\blacktriangleright$\end{remar}}

\newcommand{\exa}{\begin{examp}\rm}
\newcommand{\eexa}{$\blacktriangleright$\end{examp}}

\def\lwvec(#1 #2){\linewd 0.1
           \lvec(#1 #2)
           \linewd 0.05}

\makeatletter
\@namedef{subjclassname@2020}{\textup{2020} Mathematics Subject Classification}
\makeatother

\begin{document}

\title[Fekete--Szeg\"{o} problem in Banach spaces]{The Fekete--Szeg\"{o} problem for spirallike mappings and non-linear resolvents in Banach spaces}

\begin{abstract}
We study the Fekete--Szeg\"{o} problem on the open unit ball of a complex Banach space.
Namely, the Fekete--Szeg\"{o} inequalities are proved for the class of spirallike mappings relative to an arbitrary strongly accretive operator, and some of its subclasses.
Next, we consider families of non-linear resolvents for holomorphically accretive mappings vanishing at the origin. We solve the Fekete--Szeg\"{o} problem over these families.
\end{abstract}

\author[M. Elin]{Mark Elin}

\address{Department of Mathematics,
         Ort Braude College,
         Karmiel 21982,
         Israel}

\email{mark$\_$elin@braude.ac.il}

\author[F. Jacobzon]{Fiana Jacobzon}

\address{Department of Mathematics,
         Ort Braude College,
         Karmiel 21982,
         Israel}

\email{fiana@braude.ac.il}


\newenvironment{dedication}
  {
   \itshape             
  \raggedright         
  }
  {\par 
  }
\maketitle

\begin{dedication}
Dedicated to the memory of Professor Gabriela Kohr.
\end{dedication}


\section{Introduction}

Let $X$ be a complex Banach space equipped with the norm $\|\cdot\|$ and let $X^*$ be the dual space of $X$. We denote by $\B$ the open unit ball in $X$. For each $x \in X \setminus \{0\}$, denote
\begin{equation}\label{Tx-set}
T(x)=\left\{\ell_x \in X^*: \|\ell_x\|=1 \text{ and } \ell_x(x)=\|x\|\right\}.
\end{equation}
According to the Hahn--Banach theorem (see, for example, \cite[Theorem~3.2]{Rudin}), $T(x)$ is nonempty and may consists of a singleton (for instance, in the case of Hilbert space), or, otherwise, of infinitely many elements. Its elements $\ell_x\in T(x)$ are called support functionals at the point $x$. 

Let $Y$ be a Banach space (possibly, different from~$X$). The set of all holomorphic mappings from $\B$ into $Y$ will be denoted by $\Hol(\B,Y)$. It is well known (see, for example, \cite{Har, E-R-S-19, GK2003, R-Sbook}) that if $f \in \Hol(\B,Y)$, then for every $x_0\in\B$ and  all $x$ in some neighborhood of $x_0 \in \B$, the mapping $f$ admits the Taylor series representation:
\begin{equation}\label{Taylor-infinite}
f(x) =\sum_{n=0}^{\infty} \frac{1}{n!} D^nf(x_0)\left[(x - x_0)^n\right],
\end{equation}
where $D^nf(x_0):\prod\limits_{k=1}^{n}X \to Y$ is a bounded symmetric $n$-linear operator that is called the $n$-th  Fr\'{e}chet derivative of $f$ at $x_0$. Also we write $
D^nf(x_0)\left[(x - x_0)^n\right]$ for $D^nf(x_0)[x - x_0,\ldots, x- x_0].
$ One says that $f$ is normalized if $f(0)=0$ and $Df(0)=\Id,$ the identity operator on $X$.

Recall that a holomorphic mapping $f : \B\to X$ is called biholomorphic if the inverse $f^{-1}$ exists and is holomorphic on the image $f(\B)$.  A mapping $f \in \Hol(\B,X)$ is said to be locally biholomorphic
if for each $x \in \B$ there exists a bounded inverse for the Fr\'{e}chet derivative $Df(x)$, see \cite{E-R-S-19, GK2003}.

In the one-dimensional case, where  $X=\C$ and $\B=\D$ is the open unit disk in $\C$, one usually writes $a_n(x-x_0)^n$ instead of $\frac{1}{n!}D^nf(x)\left[(x - x_0)^n\right]$ in~\eqref{Taylor-infinite}. The classical Fekete--Szeg\"o problem \cite{F-S}  for a given subclass $\mathcal{F}\subset\Hol(\D,\C)$ is to find
\[
\sup_{f\in\mathcal{F}} \left| a_3 -\nu a_2^2\right|,\quad \mbox{where }\ f(z)=z+a_2z^2+a_3z^3+\ldots.
\]

In multi-dimensional settings various analogs of the classical Fekete--Szeg\"o problem for different classes of holomorphic mappings have been studied by many mathematicians. Nice survey of the current state of the art and references can be found in  \cite{HKK2021} and \cite{LX}.

H.~Hamada, G.~Kohr and M.~Kohr  in  \cite{HKK2021} introduced a new quadratic functional that generalizes the Fekete--Szeg\"o functional to infinite-dimensional settings.
Moreover, they estimated this functional over several classes of holomorphic mappings, including starlike mappings and non-linear resolvents of normalized holomorphically accretive mappings.

The aim of this paper is to extend the method used in  \cite{HKK2021} and solve the Fekete--Szeg\"{o} problem over the classes of spirallike mappings and resolvents of non-normalized holomorphically accretive mappings. Along the way we generalize some results in  \cite{HKK2021} and \cite{E_J-21}.

Spirallike mappings in Banach spaces were first introduced and studied in the mid 1970's by K.~Gurganus and T.~J.~Suffridge. This study has evolved into a coherent theory thanks to the influential contributions of Gabriela Kohr and her co-authors (I.~Graham, H.~Hamada, M.~Kohr and others) over the past decades (some details can be found below). As for non-linear resolvents, they seem to have been among the last issues that caught her attention. Progress on this topic is reflected in \cite{GHK2020, HKK2021}.

 \medskip

\section{Preliminaries}
\setcounter{equation}{0}

Recall that for a densely defined linear operator  $A$ with the domain $D_A \subset X$, the set $V(A)=\{\ell_x(Ax) : x \in D_A, \|x\|=1,\, \ell_x \in T(x)\}$ is called the numerical range of $A$.
\begin{defin}\label{def-accr-op}
Let $A \in L(X)$ be a bounded linear operator on $X$.
Then $A$ is called accretive if
  \[
\Re \ell_x(Ax)\ge0
\]
for all $\ x\in X\setminus\{0\}$, or, what is the same, if $m(A)\ge 0$, where $m(A):=\inf\left\{\Re\lambda: \lambda\in V(A)\right\}$. If for some $k>0$,
  \[
\Re \ell_x(Ax)\ge k\|x\|
\]
for all $x\in X\setminus\{0\}$, the operator $A$ is called strongly accretive.
\end{defin}

The notion of accretivity was extended by Harris \cite{Har} to involve holomorphic mappings (see also  \cite{R-Sbook, E-R-S-19}).
\begin{defin}\label{def-accr-op}
Let $h \in\Hol(\B,X)$. This mapping $h$ is said to be holomorphically accretive if
  \[
m(h):=\liminf_{s\to1^-}\bigl( \inf \left\{ \Re \ell_x(h(sx)): \ \|x\|=1,\ \ell_x\in T(x)\right\}\bigr)\ge0.
\]
 In the case where the last lower limit $m(h)$ is positive, $h$ is called strongly holomorphically accretive.
\end{defin}

\begin{remark}\label{rema1}
According to \cite[Proposition 2.3.2]{E-R-S-19} if $h(0)=0$ then $V(A)\subset \overline{\rm conv}\,V(h)$, where $A=Dh(0)$, in particular,  $m(A)\ge m(h)$.  Consequently, if $h$ is holomorphically accretive, its linear part at zero $A$ is accretive too. Furthermore, for such mappings Proposition~2.5.4 in \cite{E-R-S-19} implies that  $h$ is holomorphically accretive if and only if $\Re\ell_x(h(x))\ge0$ for all $\ x\in \B\setminus\{0\}$.
\end{remark}

The main feature of the class of  holomorphically accretive mappings is that they generate  semigroups of holomorphic self-mappings on $\B$, so they are of most importance in dynamical systems \cite{R-Sbook, E-R-S-19}. A very fruitful characterization of holomorphically accretive mappings is:

\begin{propo}[Theorem~7.3 in \cite{R-Sbook}, see also \cite{E-R-S-19}]\label{propo-rc}
  A mapping $h\in\Hol(\B,X)$ is holomorphically accretive if and only if it satisfies the so-called range condition (RC), that is, $(\Id+rh)(\B)\supseteq\B$ for each $r>0$, and the inverse mapping $J_r:=(\Id+rh)^{-1}$ is a well-defined holomorphic self-mapping of $\B$.
\end{propo}
The mapping $J_r$ that occurs in this proposition is called the {\it non-linear resolvent of $h$}. In other words, the non-linear resolvent is the unique solution $w=J_r(x)\in\B$ of the functional equation
\begin{equation*}
w + r h (w) = x\in\B,\quad r>0.
\end{equation*}
Assuming $h(0)=0$, one sees that $J_r(0)=0$ for all $r>0$. If, in addition, $A=Dh(0),$ then $DJ_r(0)=(\Id+rA)^{-1}$. Furthermore, the accretivity of $A$ mentioned in Remark~\ref{rema1}, implies $DJ_r(0)$ is strongly contractive because $\|(\Id+rA)^{-1}\|<1. $

We use the following classes (see \cite{GK2003} and references therein):
\begin{eqnarray*}
 \mathcal{N}\! &\!=\!&\! \left\{h\in\Hol(\B,X)\!:  h(0)=0, \Re\ell_x(h(x))>0, \ x\in \B\!\setminus\!\{0\}, \ell_x\in T(x) \right\}\!,\\
 \mathcal{M}\! &\!=\!&\! \left\{h\in\mathcal{N}, Dh(0)=\Id\right\}
\end{eqnarray*}
and (see \cite{GHKK2017})
\begin{eqnarray}\label{n_a}
  \mathcal{N}_A &:=& \left\{h\in\mathcal{N}: \  Dh(0)=A  \right\}.
\end{eqnarray}

To proceed, we note that the inclusion $h\in\mathcal{N}$ can be expressed as $\ell_x(h(x))\in g_0(\D), \ x\in \B\!\setminus\!\{0\}$, where $g_0(z)=\frac{1+z}{1-z}\,.$ At the same time, $\overline{V(A)}$ is a compact subset of the open right half-plane, hence the inclusion  $\ell_x(h(x))\in g_0(\D)$ is imprecise. It can be improved by using other functions $g\prec g_0,$ bearing in mind that $g(\D)$ should contain $V(A)$ by
Remark~\ref{rema1}.

Throughout this paper we suppose that the following conditions hold
\begin{assump}\label{assu-1}
A linear operator $A$ is bounded and strongly accretive. A~function $g=g_A \in \Hol(\D,\C)$ satisfies $g\prec g_0$ and $\overline{V(A)}\subset g(\D)$. Therefore $\Delta:=g^{-1}(V(A))$ is compactly embedded in $\D$.
\end{assump}

\begin{defin}[cf. \cite{Ch2014, XL}]\label{def-Mg-class}
Let $A$ and $g$ satisfy Assumption~\ref{assu-1}. Denote 
\begin{equation}\label{Mg-class}
\mathcal{N}_A(g)\!:=\!\left\{h \in \mathcal{N}_A\!: \frac{\ell_x(h(x))}{\|x\|}\! \in g(\D), x \in \B\!\setminus\!\{0\}, \ell_x \in T(x) \!\right\}\!.
\end{equation}
\end{defin}

We now consider specific choices of $g$ providing some properties of semigroups  generated by $h \in \mathcal{N}_A(g)$:
\begin{itemize}
  \item [(a)] $g_1^\alpha(z):=\left(\frac{1+z}{1-z}\right)^\alpha,\alpha\in(0,1)$: It can be shown that the semigroup generated by every $h\in \mathcal{N}_A(g_1^\alpha)$ can be analytically extended with respect to parameter $t$ to the sector $|\arg t|<\frac{\pi(1-\alpha)}2$; for the one-dimensional case see \cite{E-Sh-Ta2018};

  \item [(b)] $g_2^\alpha(z):=\alpha+ (1-\alpha)\frac{1+z}{1-z}, \alpha\in (0,m(A))$: it follows from Lemma~3.3.2 in \cite{E-R-S04} that the semigroup $\{u(t,x)\}_{t\ge0}$ generated by any element of $\mathcal{N}_A(g_2^\alpha)$ satisfies the estimate $\|u(t,x)\|\le e^{-t\alpha}\|x\|$ uniformly on the whole $\B$;

  \item [(c)] $g_3^\alpha(z):=\frac{1-z}{1-(2\alpha-1)z}, \alpha\in (0,1),$ maps $\D$ onto  a disk $\Delta$ tangent the imaginary axis. In a sense this choice is dual to the previous one (in the one-dimensional case such duality was investigated in \cite{B-C-DM-13});
\end{itemize}
In what follows we will refer to these functions as $g_0,g_1^\alpha, g_2^\alpha, g_3^\alpha.$

\vspace{2mm}

Another area where holomorphically accretive mappings are widely used is geometric function theory.  The study of spirallike mappings is a good example of this fruitful connection.

\begin{defin}[see \cite{STJ-77, GK2003, E-R-S04, R-Sbook}]\label{def-spiral}
  Let $A$ be a strongly accretive operator. A biholomorphic mapping $f\in \Hol(\B,X)$ is said to be spirallike relative to $A$ if its image is invariant under the action of the semigroup $\{e^{-tA}\}_{t\ge0}$, that is, $e^{-tA}f(x)\in f(\B)$ for all $t\ge0$ and $x\in\B$. The set of all spirallike mappings relative to $A$ is denoted by $\s_A(\B)$.

  If  $f$ is spirallike relative to $A=e^{-i \beta}\Id$  for some $|\beta|<\frac{\pi}{2}$, then $f$ is said to be spirallike of type $\beta$. In the particular case where $\beta=0$, spirallike mappings relative $A=\Id$ are called starlike.
\end{defin}
The following result is well known (see, for example, Proposition~2.5.3 in \cite{E-R-S04} and references therein).
\begin{propo}\label{propo1}
  Let $A\in L(X)$ be strongly accretive, and let $f\in\Hol(\B,X)$ be a normalized and locally biholomorphic mapping. Then $f\in \widehat{S}_A(\B)$ if and only if the mapping $h:=(Df)^{-1}Af$ belongs to $\mathcal{N}_A$.
\end{propo}
This proposition {\it inter alia} implies that a spirallike mapping $f$ relative to $A$ linearizes the semigroup $u(t,x)$ generated by $h=(Df)^{-1}Af$ in the sense that $f\circ u(t,f^{-1}(x))=e^{-tA}x$ on $f(\B)$. In the one-dimensional case, any linear operator is scalar, hence can be chosen to be  $A=e^{i\beta}\Id$. In this case the inclusion $h=(Df)^{-1}Af\in \mathcal{N}_A$ is equivalent to $\Re \left(e^{-i \beta} \frac{zf'(z)}{f(z)}\right)>0$. This is the standard definition of spirallike functions of type $\beta$ on $\D$ (see, for example,~\cite{Dur, GK2003}).

Moreover, according to Proposition \ref{propo1}, it is relevant to consider biholomorphic functions $g\in\Hol(\D,\C)$ satisfying Assumption \ref{assu-1}  and to distinguish subclasses of $\widehat{S}_A(\B)$ letting
\begin{equation}\label{spiral-g}
  \s_g(\B):=\left\{ f\in\widehat{S}_A(\B): (Df)^{-1}Af\in \mathcal{N}_A(g)\right\}.
\end{equation}
In particular, $\widehat{S}_{g_0}(\B)=\widehat{S}_A(\B)$. Further, $\widehat{S}_{g_1^\alpha}(\B)$ consists of mappings that are spirallike relative to operator $e^{i\beta}A$ with  any $|\beta|<1-\alpha$. The classes $\widehat{S}_{g_2^\alpha}(\B)$ and $\widehat{S}_{g_3^\alpha}(\B)$ are also of specific interest. For instance, if $A=e^{i\beta}\Id$ and $\alpha=\lambda\cos\beta$, the class $\widehat{S}_{g_3^\alpha}(\B)$ of spirallike mappings of type $\beta$ of order $\lambda$ is a widely studied object. The intersection $\widehat{S}_{g_2^\alpha}(\B)\bigcap \widehat{S}_{g_3^\alpha}(\B)$ consists of strongly spirallike mappings (for an equivalent definition and properties of these mappings see \cite{H-K-2001a, H-K-2001, CKK}).

\medskip

\section{Auxiliary lemmata}\label{sect-lemmas}
\setcounter{equation}{0}

Our first auxiliary result essentially coincides with Theorem~2.12 in \cite{HKK2021}. We present it in a somewhat more general form.

\begin{lemma}\label{thm-K}
 Let $p(z) =a + p_1z + p_2z^2 +o(z^2)$ and $\phi(z ) =a + b_1z+ b_2z^2 +o(z^2)$ be holomorphic functions on $\D$ such that $\phi \prec p$. Then for every $\mu\in\C$ the following sharp inequality holds:
\[
|b_2 - \mu b_1^2|  \le \max \left( |p_1|, |p_2 - \mu p_1^2|\right).
\]
\end{lemma}
\begin{proof}
  Since $\phi \prec p$, there is a function $\omega\in\Omega$ such that $\phi=p\circ\omega$. Let $\omega(z)=c_1z+c_2z^2+o(z^2)$. Then
  \[
  b_1=p_1c_1\quad\mbox{and}\quad b_2=p_2c_1^2+p_1c_2.
  \]
Therefore
\[
b_2-\mu b_1^2=(p_2-\mu p_1^2)c_1^2+p_1c_2.
\]
Because the inequality $|c_2|\le 1-|c_1|^2$ holds and is sharp (see, for example, \cite{Dur}), one concludes that $|b_2-\mu b_1^2|$ is bounded by a convex hull of $|p_1|$ and $|p_2-\mu p_1^2|$. The result follows.
\end{proof}

\begin{lemma}\label{lem-2}
 Let $h\in\Hol(\B,X)$ with $h(0)=0$ and $B\in L(X)$ with $\rho:=\|B\|\le1$. For any $x\in\partial\B$ and $\ell\in X^*$ denote
 \begin{equation*}\label{fi-wave}
\varphi(t):=  \frac{\ell\left(h(tBx)\right)}t,\qquad  t \in \D\setminus \{0\}.
\end{equation*}
\begin{itemize}
  \item [(i)] The function $\varphi$ can be analytically extended to the disk $\frac1\rho\D$ with  the Taylor expansion $\varphi(t)=b_0+b_1t+b_2t^2+o(t^2)$, where $b_0= \ell(Dh(0)Bx)$,
 \begin{equation}\label{b1-2by-f}
b_1=\frac{1}{2!}\, \ell\left(D^2h(0)[(Bx)^2]\right) \quad \mbox{and} \quad b_2 =\frac{1}{3!} \, \ell\left(D^3h(0)[(Bx)^3]\right).
\end{equation}
  \item [(ii)] If, in addition, $\ell\in T(Bx)$ and $h\in\mathcal{N}_A(g)$, then $\varphi(\D)\subset \rho \widehat{g}(\rho\D)$, where $\widehat{g}(t)= g\left(\frac{\tau-t} {1- t \overline{\tau}}\right)$ and $\tau=g^{-1}\left(\frac{\ell(Dh(0)Bx)}{\|Bx\|}\right)$.
\end{itemize}
\end{lemma}

   \begin{proof}
     The function $\varphi$ is holomorphic whenever $\|tBx\|<1$, that is, for $|t|<\frac1{\rho}\le\frac1{\|Bx\|}$. Represent $h$ by the Taylor series \eqref{Taylor-infinite}. A straightforward calculation proves (i).

    Recall that  $h\in\mathcal{N}_A(g)$,  hence Definition~\ref{def-Mg-class} implies $\frac{\varphi(t)}{\|Bx\|}\in g(\D)=\widehat{g}(\D)$ as $|t|<\frac1\rho$. Therefore the function $\widehat{g}^{-1}(\frac{\varphi(\cdot)}{\|Bx\|})$ maps the disk of radius $\frac1\rho$ into $\D$ and preserves zero. By the Schwarz Lemma $\widehat{g}^{-1}(\frac{\varphi(t)}{\|Bx\|})\le\rho|t|$. Thus $\varphi\prec \|Bx\|\widehat{g}(\rho\,\cdot)$. The proof is complete.
     \end{proof}

 A mapping $f \in \Hol(\B, X)$ is said to be of one-dimensional type if it takes the form  $f(x)=s(x)x$ for some $s\in\Hol(\B,\C)$. Such mappings were studied by many authors (see, for example, \cite{Lic-86, ES2004, D-L-21} and references therein).

\begin{lemma}\label{lemm-F-S-norm}
Let $f\in\Hol(\B,X)$ be a mapping of one-dimensional type. Then for every $n\in\N$ the entire mapping $x\mapsto D^nf(0)[x^n]$ is  also of one-dimensional type. Therefore for any $x\in\partial\B,\ \ell_x\in T(x)$ and constants $\mu_j\in \C,\ j=1,2,\ldots,$ we have
\begin{eqnarray*}\label{1dim}
\left|\ell_x\left( \sum_{j=1}^n \mu_j D^jf(0)[x^j]\right) \right| =\left\|   \sum_{j=1}^n  \mu_j D^nf(0)[x^j]      \right\| .
\end{eqnarray*}
\end{lemma}

\begin{proof} The first assertion is evident (for detailed calculation see \cite{E-J-22a}). To prove the second one we note that there is a function $F\in\Hol(X,\C)$ such that  $\sum_{j=1}^n \mu_j D^jf0)[x^j] =F(x)x$. Thus for any $x\in\partial\B$ we have
  \begin{eqnarray*}\label{DF}
&&\left\|   \sum_{j=1}^n \mu_j D^jf(0)[x^j]  \right\|=|F(x)|\|x\|\qquad \mbox{and}  \\
&& \ell_x\left( \sum_{j=1}^n \mu_j D^jf(0)[x^j]\right) =F(x)\ell_x(x)=F(x),
\end{eqnarray*}
which completes the proof. 
\end{proof}

\medskip

\section{Fekete--Szeg\"{o} inequalities for spirallike mappings}\label{sect-spiral}
\setcounter{equation}{0}

In what follows $A$ and $g$  satisfy Assumption~\ref{assu-1}, and the class $\s_{g}(\B)$ is defined by formula \eqref{spiral-g}.

 \begin{theorem}\label{thm-main1}
  Let $x \in \partial \B$,  $\ell_x \in T(x)$ and  $\tau=g^{-1}(\ell_x(Ax))$. Assume that $g\left(\frac{\tau-z}{1-z\overline{\tau}}\right) =q_0+q_1z+q_2z^2+o(z^2)$.
  Given $f \in \Hol (\B, X)$ denote
\begin{eqnarray}\label{coeff-a223}
   \nonumber  \widetilde{a}_2^2 &=& \frac12  \ell_x\left( D^2f(0)\left[x,D^2f(0)[x,Ax]\right]  - \frac12 D^2f(0)\left[x,AD^2f(0)[x^2]\right]   \right),         \\
   a_2 &=& \frac1{2!} \ell_x\left(2D^2f(0)[x,Ax] - AD^2f(0)[x^2]\right), \\
\nonumber    a_3 &=&\frac1{2\cdot3!} \ell_x \left( 3 D^3f(0)[x^2,Ax] - AD^3f(0)[x^3] \right).
\end{eqnarray}
If $f \in\s_g(\B)$, then for any $\nu \in\C$ we have
\begin{equation}\label{main1-FS}
\left|a_3-(\nu -1)a_2^2- \widetilde{a}_2^2\right|\leq \frac{|q_1|}{2}\max\left\{1 ,\left|\frac{q_2}{q_1} +2(\nu-1)q_1\right| \right\}.
\end{equation}
\end{theorem}

\begin{remark}\label{rem-new1}
It can be directly calculated that $q_1=-g'(\tau)(1-|\tau|^2)$ and
 $\frac{q_2}{q_1}=\overline\tau- \frac{g''(\tau)}{2g'(\tau)}(1-|\tau|^2)$. Thus the right-hand side in \eqref{main1-FS} can be expressed by the hyperbolic and pre-Schwarzian derivatives of $g$.
\end{remark}

\begin{proof}
Let $h(x) =\left[Df(x)\right]^{-1}Af(x)$.
Recall that $f$ is a normalized biholomorphic mapping.  Let the Taylor expansion of $f$ be
\begin{equation}\label{tay-f}
f(x)=x+\frac{1}{2!}D^2f(0)[x^2]+\frac{1}{3!}D^3f(0)[x^3]+o(\|x\|^3),
\end{equation}
so that
\begin{equation}\label{DF1}
 Df(x)[w]=w+D^2f(0)[x,w]+\frac{1}{2}D^3f(0)[x^2,w]+o(\|x\|^2).
\end{equation}
Take the Taylor expansion ${h(z)=Ax+\frac{1}{2}D^2h(0)[x^2]+\frac{1}{6}D^3h(0)[x^3]+o(\|x\|^3)}$ and substitute it together with \eqref{tay-f}--\eqref{DF1} into the equality
\begin{equation*}
Df(x)[h(x)]=Af(x).
\end{equation*}

This gives us
\label{Dfh}
\begin{eqnarray*}
&& Ax+\frac{1}{2}D^2h(0)[x^2]+\frac{1}{6}D^3h(0)[x^3]+D^2f(0)[x,Ax]\\
&+&\frac{1}{2}D^2f(0)[x,D^2h(0)x^2]+\frac{1}{2}D^3f(0)[x^2,Ax]+o(\|x\|^3)\\
&=&Ax+\frac{1}{2}AD^2f(0)[x^2]+\frac{1}{6}AD^3f(0)[x^3]+o(\|x\|^3).
\end{eqnarray*}
Equating terms of the same order leads to
\begin{equation*}\label{2-nd-ord}
 \frac{1}{2}D^2h(0)[x^2]+ D^2f(0)[x,Ax]=\frac{1}{2}AD^2f(0)[x^2]
\end{equation*}
and
\begin{equation*}\label{3-rd-ord}
\frac{1}{6}D^3h(0)[x^3]+\frac{1}{2}D^2f(0)[x,D^2h(0)x^2]+\frac{1}{2}D^3f(0)[x^2,Ax] =\frac{1}{6}AD^3f(0)[x^3].
\end{equation*}
In turn, these equalities imply
\begin{equation*}
D^2h(0)[x^2]=AD^2f(0)[x^2]-2D^2f(0)[x,Ax]
\end{equation*}
and
\begin{eqnarray*}
\nonumber&&D^3h(0)[x^3]=AD^3f(0)[x^3]-3D^2f(0)[x,D^2h(0)x^2]-3D^3f(0)[x^2,Ax]\\
\nonumber&&=AD^3f(0)[x^3]-3D^3f(0)[x^2,Ax]\\
&&-3D^2f(0)\left[x,AD^2f(0)[x^2]\right]+6D^2f(0)\left[x,D^2f(0)[x,Ax]\right]\!.
\end{eqnarray*}

Recall that $\ell_x(Ax) \in V(A)\subset g(\D)$, so $\tau \in \Delta$ is well-defined. Similarly to the proof of the Theorem 3.1 in \cite{HKK2021}, denote
\begin{equation*}
\varphi(t)=\left\{\begin{array}{ll}
                    \frac{\ell_x\left(h(tx)\right)}t, &  t \in \D\setminus \{0\},\vspace{2mm} \\
                    \ell_x(Ax), & t=0.
                  \end{array}
\right.
\end{equation*}
Then $\varphi\in\Hol(\D,\C)$  by assertion (i) of  Lemma~\ref{lem-2} with $B=\Id$,
\begin{equation*}
b_1=\frac{1}{2!}\ell_x\left(D^2h(0)[x^2]\right) \quad\mbox{and}\quad b_2 =\frac{1}{3!}\ell_x\left( D^3h(0)[x^3]\right).
\end{equation*}
Using $a_2,\,\widetilde{a}_2^2$ and $a_3$ defined in \eqref{coeff-a223} we get
\[
b_1=-a_2 \quad\mbox{and}\quad b_2 =2\widetilde{a}_2^2 -2a_3.
\]
Therefore, 
\begin{eqnarray*}
  \left|a_3 -  \widetilde{a}_2^2 -(\nu-1)a_2^2\right| = \frac12\left|b_2- 2 (1-\nu)b_1^2\right| .
\end{eqnarray*}

Also, by assertion (ii) of the same  Lemma~\ref{lem-2}, $\varphi\prec \widehat{g},\ \widehat{g}(t)=g(\frac{\tau-t}{1-\overline\tau t})$.

 To this end we apply Lemma~\ref{thm-K}  with $p =\widehat g $ and $\mu = 2(1-\nu)$ and obtain  estimate~\eqref{main1-FS}.
 \end{proof}

There are two ways to make the above result more explicit: to fix some concrete forms of the function $g$, or to put additional restrictions on the mapping~$f$. We start with some concrete choices of $g$.

Recall that for every strongly accretive operator $A$ and every spirallike mapping $f$ relative to $A$, the mapping $h:=\left(Df\right)^{-1} Af$ is holomorphically accretive. Hence one can always choose $g=g_0$, where $g_0(z)=\frac{1+z}{1-z}$ is defined above. Denoting $\ell:=\ell_x(Ax)$ and using Remark~\ref{rem-new1}, we conclude that every spirallike mapping relative to $A$ satisfies
\begin{eqnarray}\label{gener-estim}
\left|a_3-(\nu -1)a_2^2- \widetilde{a}_2^2\right| \leq
\Re\ell \cdot \max\left(1, \left|1+4(\nu-1)  \Re\ell \right| \right) .
\end{eqnarray}
In the one-dimensional case, this inequality coincides with the result of Theorem~1 in \cite{Ke-Me} for $\lambda=0$. 
By choosing other $g\prec g_0$ functions and denoting $\ell:=\ell_x(Ax)$ as above, more precise estimates can be obtained.

Assume, for example, that  $\ell_x(h(x))$ belongs to some sector of the form $\left\{ w: |\arg w|<\frac{\pi \alpha}{2} \right\}$, $\alpha\in(0,1),$ for all $x\in\B$, where $h=\left(Df\right)^{-1}Af$. Then one can set $g=g_1^\alpha$ and to get

\begin{corol}\label{corol-angle}
Every  $f \in\s_{g_1^\alpha}(\B)$ satisfies
\[
\left|a_3-(\nu -1)a_2^2- \widetilde{a}_2^2\right|\leq \alpha|\ell|\cos\arg \ell^{\frac1\alpha} \cdot\max\left\{1 ,Q_{1,\alpha} \right\},
\]
where $Q_{1,\alpha} =\Re \ell^{\frac1\alpha} \left| 4\alpha(\nu-1) \ell^{\frac{\alpha-1}\alpha}  + \frac{1}{\ell^{\frac1\alpha}} \left(\alpha+ i\tan\arg \ell^{\frac1\alpha} \right) \right|$.
\end{corol}

Also assuming that $\frac{\ell_x(h(x))}{\|x\|}$ is bounded away from the imaginary axis, namely, $\Re \frac{\ell_x(h(x))}{\|x\|}>\alpha,\ \alpha\in(0,1)$, we choose $g=g_2^\alpha$. In this situation, we have

\begin{corol}\label{corol-boun_away}
Every  $f \in\s_{g_2^\alpha}(\B)$ satisfies
\begin{eqnarray*}
\left|a_3-(\nu -1)a_2^2- \widetilde{a}_2^2\right| \leq
\Re\ell \cdot \max\left\{1,Q_{2,\alpha }  \right\},
\end{eqnarray*}
where $Q_{2,\alpha }=\left|1+4(\nu-1)(1-\alpha)  \Re\ell \right| .$
\end{corol}
In particular, taking $\alpha=0$, we return to inequality~\eqref{gener-estim} for all spirallike mappings relative to the linear operator $A$.

Another interesting (and, as we mentioned, dual) case occurs when $\frac{\ell_x(h(x))}{\|x\|}$ lies in some circle tangent to the imaginary axis. We can then set  $g=g_3^\alpha$.

\begin{corol}\label{corol-order}
Every  $f \in\s_{g_3^\alpha}(\B)$ satisfies
\begin{eqnarray*}
\left|a_3-(\nu -1)a_2^2- \widetilde{a}_2^2\right| \leq (\Re\ell -|\ell|^2\alpha) \cdot \max\left\{1,Q_{3,\alpha}  \right\},
\end{eqnarray*}
where $Q_{3,\alpha}=\left|1-2\overline{\ell}\alpha + 4(\nu-1)(\Re\ell -|\ell|^2\alpha) \right| .$
\end{corol}
Recall that for $A=e^{i\beta}\Id$, the class $\s_{g_3^\alpha}(\B)$ consists of so-called spirallike mappings of type $\beta$ of order $\alpha$.

\begin{remark}
It is worth mentioning that even for  the the case in which $A$ is a scalar operator, the estimates above (starting from \eqref{gener-estim}) are new. Since the class of spirallike mappings contains the class of starlike mappings, these estimates generalize Corollary 3.4 (i)--(iv) in \cite{HKK2021} for starlike mappings.
\end{remark}

In the rest of this section we deal with mappings $f$ that satisfy:

\begin{assump}\label{assu-2}
There exists a function $\kappa:\partial \B\to \C$ such that
\begin{equation}\label{delta}
D^2f(0)[x^2]=\kappa(x)x,\quad x \in \partial\B.
\end{equation}
The Fr\'{e}chet derivatives of $f$ of second and third order $D^2f(0)$ and $D^3f(0)$ commute with the linear operator $A$ in the sense that
\begin{equation}\label{cond1-commut}
D^kf(0)[x^{k-1},Ax]=AD^kf(0)[x^k], \quad k=2,3.
\end{equation}
\end{assump}
Condition \eqref{delta} holds automatically for one-dimensional type mappings (spirallike mappings of one-dimensional type were studied, for instance,  in \cite{ES2004, LX, E-J-22a}), while condition \eqref{cond1-commut} holds automatically whenever $A$ is a scalar operator.

In turn, relations \eqref{cond1-commut} in Assumption~\ref{assu-2} imply  that formulae~\eqref{coeff-a223} become
\begin{eqnarray}\label{a223-commute}
  a_2 &=& \frac1{2!}\ell_x\left( AD^2f(0)[x^2]\right),\nonumber\\
 \widetilde{a}_2^2 &=& \frac14  \ell_x \left( AD^2f(0)[x, D^2f(0)[x^2]] \right),       \\
 a_3 &=&\frac1{3!} \ell_x \left(AD^3f(0)[x^3] \right)\nonumber  .
\end{eqnarray}

 \begin{corol}\label{corol_generalize}
If $f\in\s_A(\B)$ satisfies Assumption~\ref{assu-2}, then for any $\nu\in\C,$
\begin{equation}\label{thm-main1-F-S}
\left|a_3- \left(\nu -1+ \frac{1}{\ell_x\left(A x\right)}\right) a_2^2\right| \leq  \frac{|q_1|}{2}\max\left\{1 ,\left|\frac{q_2}{q_1} +2(\nu-1)q_1\right| \right\}.
\end{equation}
\end{corol}

\begin{proof}
  Indeed, denote  $ \alpha=\ell_x\left(A x\right)$.Then $a_2=\frac12 \kappa(x)\alpha$ and
\begin{eqnarray*}
\nonumber\widetilde{a}_2^2 &=& \frac14  \ell_x \left( AD^2f(0)[x,\kappa(x)x] \right)=  \frac14\cdot \kappa(x) \ell_x \left( AD^2f(0)[x^2] \right) \\
\nonumber&=& \frac14\cdot \kappa(x) \ell_x \left( A\kappa(x)x  \right) =\frac{\alpha} 4\cdot (\kappa(x))^2.
\end{eqnarray*}
Thus $\widetilde{a}_2^2 =\frac1\alpha a_2^2$ and hence
\begin{equation*}
|a_3-(\nu -1)a_2^2- \widetilde{a}_2^2|=\left|a_3-\left(\nu -1+ \frac{1}{\alpha}\right)a_2^2\right|.
\end{equation*}
So, estimate \eqref{thm-main1-F-S} follows from Theorem~\ref{thm-main1}.
\end{proof}

Let $A$ be a scalar operator. Without loss of generality, we assume $A=e^{i\beta}\Id,\ |\beta|<\frac\pi2$. Then it follows from Assumption~\ref{assu-2} that formulae~\eqref{coeff-a223} (or \eqref{a223-commute}) become
 \begin{equation*}
  \begin{array}{ll}
    \displaystyle a_2= \frac1{2!}\kappa(x)e^{i\beta} \quad
    \displaystyle  \widetilde{a}_2^2= \left(\frac{1}{2!}\kappa(x)\right)^2e^{i\beta},   \quad
     \displaystyle a_3=  \frac{1}{3!} \ell_x \left( D^3f(0)[x^3] \right)e^{i\beta}.
  \end{array}
\end{equation*}

These relations and Lemma~\ref{lemm-F-S-norm} imply immediately
\begin{corol}
If $f\in\Hol(\B,X)$ is a spirallike mapping of type $\beta$, that satisfies Assumption~\ref{assu-2}. Then for any $\mu \in \C$ we have
\begin{equation*}
\left|a_3-\mu a_2^2\right|\leq \frac{|q_1|}{2}\max\left\{1 ,\left|\frac{q_2}{q_1} +2(\mu-e^{-i\beta})q_1\right| \right\}.
\end{equation*}
If, in addition, $f$ is of one-dimensional type, then for any $x\in \partial \B$ we have
\begin{eqnarray*}
\left\|\frac1{3!}D^3f(0)[x^3]-\mu \cdot\frac1{2!} D^2f(0)\left[x,\frac1{2!}D^2f(0)[x^2]\right]\right\|\\
\leq \frac{|q_1|}{2}\max\left\{1 ,\left|\frac{q_2}{q_1} +2(\mu-e^{-i\beta})q_1\right| \right\}.
\end{eqnarray*}
 \end{corol}
The last estimate coincides with Theorem 2 in \cite{E-J-22a}.

\medskip

\section{Fekete--Szeg\"{o} inequalities for normalized non-linear resolvents}  \label{sect-resolv}

\setcounter{equation}{0}

 As above, we suppose that $A \in L(X)$ and $g\in\Hol(\D,\C)$ satisfy Assumption~\ref{assu-1} and $h\in \mathcal{N}_A(g)$. In this section we concentrate on the non-linear resolvent $J_r:=(\Id+rh)^{-1},\ r>0,$ that is well-defined self-mappings of the open unit ball $\B$ that solves the functional equation
 \begin{equation}\label{resolv-eq}
J_r(x) + r h (J_r(x)) = x \in\B, \quad r>0.
\end{equation}

\begin{lemma}\label{lemm-resolv}
{\ }

\begin{itemize}
  \item [(a)] For any $r>0$,  the operator $B_r:=DJ_r(0)=(\Id+rA)^{-1}$ is strongly contractive, that is, $\rho_r:=\| B_r\| <1$.
  \item [(b)] If $h$ is of one-dimensional type, then $A$ is a scalar operator and $J_r$, $r>0$, is of one-dimensional type too.
\end{itemize}
\end{lemma}
\begin{proof}
 Assertion (a) follows from the strong accretivity of $A$.

Since $h$ is of one-dimensional type, it has the form $h(x)=s(x)x$, where $s\in \Hol(\B,\C)$.
Therefore $A=Dh(0)=s(0)\Id$.
In addition, \eqref{resolv-eq} implies
\begin{equation*}\label{jj}
x=J_r(x)+rs(J_r(x))J_r(x)=(1+rs(J_r(x)))J_r(x),
\end{equation*}
that is, $J_r(x)$ is collinear to $x$.
\end{proof}

Further, it is natural to consider the family of normalized resolvents $(\Id+rA)J_r$ and to study the Fekete--Szeg\"o problem for these mappings.

We now present the main result of this section.

\begin{theorem}
Let $h \in \mathcal{N}_A(g)$ and $J_r$ be the nonlinear resolvent of $h$ for some $r>0$.
For $x \in \partial{\B}$  and  $\ell_r:=\ell_{B_rx} \in T(B_rx)$, let
\begin{eqnarray}\label{a1-2-3res-def}
\nonumber \widetilde{a}_2^2&:=&\ell_{r}\left((\Id+rA)\frac{1}{2!}D^2J_r(0)\left[x,(\Id+rA)\frac{1}{2!}D^2J_r(0)[x^2]\right]\right),\\
a_2&:=&\ell_{r}\left((\Id+rA)\frac{1}{2!}D^2J_r(0)[x^2]\right),\\
\nonumber a_3&:=&\ell_{r}\left((\Id+rA)\frac{1}{3!} D^3J_r(0)[x^3] \right).
\end{eqnarray}
Then for any $\nu \in \C$ we have
\begin{equation}\label{F--S-res}
\left|a_3-2\widetilde{a}_2^2-(\nu-2)a_2^2\right|\leq r |q_1| \|B_rx\|\rho_r^2\max \left(1, Q_r(x)\right),
\end{equation}
where
\begin{equation}\label{Qr}
Q_r(x):=\left|\frac{q_2}{q_1}-(2-\nu)r q_1\|B_rx\|\right|
\end{equation}
and $q_1, q_2$  are the Taylor coefficients of  $\widehat{g}(t)=g\left(\frac{\tau-t}{1-t\overline{\tau}}\right)$ with  $\tau=g^{-1}\left(\frac{\ell_{r}(AB_rx)}{\|B_rx\|}\right)$.
\end{theorem}

\begin{proof}
Denote $x_r:=B_rx$. Using the functional equation~\eqref{resolv-eq}, one finds
\begin{equation*}\label{D2Fres-0}
\left(I+rA\right)D^2J_r(0)[x,y]=-rD^2h (0)\left[x_r,B_ry\right]
\end{equation*}
and
\begin{eqnarray*}\label{D2-3res-0}
&&(\Id+rA)\frac{1}{2!}D^2J_r(0)[x^2]=-r \frac{1}{2!}D^2h (0)\left[(x_r)^2\right]\!,\\
&&(\Id+rA)\frac{1}{3!}D^3J_r(0)[x^3]=-r\frac{1}{3!} B_rD^3h (0)\left[(x_r)^3\right]\\
\nonumber&&\hspace{11mm} +2 r^2\cdot\frac{1}{2!} B_rD^2h (0)\left[x_r,B_r\frac{1}{2!}D^2h (0)\left[(x_r)^2\right]\right]\!\!.
\end{eqnarray*}
Thus the quantities $a_2,\widetilde{a}_2^2$ and $a_3$  defined by \eqref{a1-2-3res-def} can be expressed by the Fr\'{e}chet derivatives of $h$:
\begin{eqnarray}\label{a-1-2-3res}
\nonumber \widetilde{a}_2^2&=&r^2\frac{1}{2!} \ell_r\left(D^2h (0)\left[x_r,\frac{1}{2!}B_rD^2h (0)\left[(x_r)^2\right]\right] \right)\\
a_2&=&-r \frac{1}{2!}\ell_r\left(D^2h (0)\left[(x_r)^2\right]\right) \\
\nonumber a_3&=&-r\displaystyle\frac{1}{3!}\ell_r\left( D^3h (0)\left[(x_r)^3\right]\right)\\
\nonumber &+&2r^2\ell_r\left(\frac{1}{2!} D^2h (0)\left[x_r,\frac{1}{2!}B_rD^2h (0)\left[(x_r)^2\right]\right]\right).
\end{eqnarray}

Denote
\begin{equation*}\label{fi-wave}
\varphi(t)=\left\{\begin{array}{ll}
                    \frac{\ell_r\left(h(tx_r)\right)}t, &  t \in \D\setminus \{0\},\vspace{2mm} \\
                    \ell_{r}(Ax_r), & t=0.
                  \end{array}
\right.
\end{equation*}
By assertion (i) of  Lemma~\ref{lem-2} with $B=B_r$, the function $\varphi$ is analytic in the disk of radius $\frac{1}{\rho_r}$  and
\begin{equation}\label{b1-2by-f}
b_1=\frac{1}{2!} \cdot\ell_r\left(D^2h(0)[(x_r)^2]\right) \quad \mbox{and} \quad b_2 =\frac{1}{3!} \cdot\ell_r\left(D^3h(0)[(x_r)^3]\right).
\end{equation}
Comparing formulae \eqref{b1-2by-f} and \eqref{a-1-2-3res} we see that
\begin{equation*}\label{b1-2by-J}
b_1=-\frac{1}{r}a_2 \qquad \text{and} \qquad
b_2 =-\frac{1}{r}(a_3-2\widetilde{a}_2^2).
\end{equation*}
Therefore, 
\begin{equation*}
  \left|a_3 -  \widetilde{a}_2^2 -(\nu-2)a_2^2\right| = r\left|b_2- r (2-\nu)b_1^2\right| .
\end{equation*}

Also, by assertion (ii) of Lemma~\ref{lem-2}, $\varphi\prec \|x_r\|\widehat{g}(\rho_r \cdot)$.

 To complete the proof we apply Lemma~\ref{thm-K}  with $p =\|x_r\|\widehat{g}(\rho_r \cdot)$ and $\mu = r(2-\nu)$.
\end{proof}

From now on, for any $x \in \partial\B$ we will adopt the notations $x_r=B_rx$ and $\ell_r:=\ell_{x_r}\in T(x_r)$.  To compare our results with the previous ones we consider some special cases.

If, for example, $A=\lambda \Id$, $\Re \lambda>0$, is a scalar operator, then $B_r=\frac{1}{1+\lambda r}\Id$, $x_r=\frac{1}{1+\lambda r} x$ and $\rho_r=\|x_r\|=\frac{1}{|1+\lambda r|}$. Thus
\begin{equation}\label{tau1}
\tau=g^{-1}\left(\frac{\ell_{r}(\lambda x_r)}{\|x_r\|}\right)=g^{-1}(\lambda).
\end{equation}
Thus inequality \eqref{F--S-res} takes the form
\begin{equation}\label{F--S-res-sc-q}
\left|a_3-2\widetilde{a}_2^2-(\nu-2)a_2^2\right|\leq \frac{|q_1|r}{|1+\lambda r|^3}\max \left(1, \left|\frac{q_2}{q_1}-\frac{q_1 r}{|1+\lambda r|}(2-\nu)\right|\right)\!,
\end{equation}
where $q_1, q_2$  are the Taylor coefficients of  $\widehat{g}(t)=g\left(\frac{\tau-t}{1-t\overline{\tau}}\right)$ with  $\tau=g^{-1}\left(\lambda \right)$.
\begin{corol}\label{corol-scalarA}
Assume that $A=\lambda \Id$, $\Re \lambda>0$ and $g=g_0$. Then for any $\nu \in \C$ we have
\begin{equation}\label{F-S-res-scalar}
\left|a_3-2\widetilde{a}_2^2-(\nu-2)a_2^2\right|\leq \frac{|1+\lambda^2|r}{|1+\lambda r|^3}\max \left(1, \left|\lambda-(2-\nu)r\frac{1+\lambda^2}{|1+\lambda r|}\right|\right)\!.
\end{equation}
\end{corol}
\begin{proof}
Since $g=g_0$,  formula~\eqref{tau1} is $\tau=g^{-1}(\lambda)=\frac{\lambda -1}{\lambda+1}.$
Thus $q_1=-(1+\lambda^2)$ and $q_2=\lambda (1+\lambda^2)$. Then \eqref{F-S-res-scalar} follows from \eqref{F--S-res-sc-q}.
\end{proof}

For $A=\Id$, Corollary~\ref{corol-scalarA} coincides with \cite[Theorem~5.6]{HKK2021}.

Another interesting case occurs when $h$ satisfies Assumption~\ref{assu-2}.

\begin{corol}\label{corol-h-assu2}
If $h\in \mathcal{N}_A(g)$ satisfies Assumption~\ref{assu-2}, then
\begin{equation}\label{F-S-1dim}
\left|a_3-(\nu-2+2\delta) a_2^2\right|\leq  r |q_1|\|x_r\|\rho_r^2\max \left(1, Q_r(x)\right),
\end{equation}
where $Q_r(x)$ is defined by \eqref{Qr} and $\delta=\frac{\ell_r(B_rx_r)}{\|x_r\|^2}$.
\end{corol}
\begin{proof}
 Since $h$ satisfies condition~\eqref{delta}, there exists a function $\kappa:\partial \B\to \C$ such that
$D^2h(0)[x^2]=\kappa(x)x,$  $x \in \partial\B$. Thus,
\begin{eqnarray*}
  a_2^2&=&\frac{r^2 }{4}\left(\ell_r\left(D^2h (0)\left[(x_r)^2\right]\right)\right)^2 = \frac{r^2 }{4}\left(\ell_{r}\left(\kappa(x_r)x_r\right)\right)^2 \\
&=&\left(\frac{r}{2}\kappa(x_r)\right)^2\|x_r\|^2.
\end{eqnarray*}

At the same time,
\begin{eqnarray*}\label{ass2-res1}
\nonumber \widetilde{a}_2^2&=&r^2\frac{1}{2!} \ell_r\left(D^2h (0)\left[x_r,\frac{1}{2!}B_rD^2h (0)\left[(x_r)^2\right]\right] \right)\\
&=&r^2\frac{1}{4}\ell_r\left(D^2h (0)\left[x_r,B_r\kappa(x_r) x_r\right] \right)\\
&=&r^2\frac{1}{4}\kappa(x_r)\ell_r\left(D^2h (0)\left[x_r,B_rx_r\right] \right).
\end{eqnarray*}
The mapping $h$ also satisfies condition~\eqref{cond1-commut}, then
\begin{eqnarray*}\label{ass2-res2}
 \widetilde{a}_2^2&=&r^2\frac{1}{4}\kappa(x_r)\ell_r\left(B_rD^2h (0)\left[(x_r)^2\right] \right)\\
&=&r^2\frac{1}{4}\kappa(x_r)\ell_r\left(B_r\kappa(x_r)x_r \right)=\left(\frac{r}{2}\kappa(x_r)\right)^2\ell_r\left(B_rx_r \right).
\end{eqnarray*}
Now estimate~\eqref{F-S-1dim} follows from the relation  $\widetilde{a}_2^2=\delta a_2^2$ with $\delta=\frac{\ell_r(B_rx_r)}{\|x_r\|^2}$.
\end{proof}

If $h$ is of a one-dimensional type, then $A=\lambda \Id$ for some $\lambda \in \C$ by Lemma~\ref{lemm-resolv}. In this case formula~\eqref{F-S-1dim} gets a simpler form.
\begin{corol}
If $h \in \mathcal{N}_A(g)$ is one-dimensional type with  $A=\lambda \Id$, then for any  $\nu \in \C$ we have
\begin{eqnarray*}
   && \left\|  (\Id+rA)\frac{1}{3!} D^3J_r(0)[x^3] -\mu (\Id+rA)\frac{1}{2!}D^2J_r(0)\left[x,(\Id+rA)\frac{1}{2!}D^2J_r(0)[x^2]\right]   \right\| \vspace{1mm} \\
   &&= \left|a_3-\mu a_2^2\right|\, \le \, \frac{r|q_1|}{|1+\lambda r|^3}\cdot \max \left( 1, \left|\frac{q_2}{q_1}- (2\delta-\mu)\frac{r q_1}{|1+\lambda r|}\right|\right)
\end{eqnarray*}
with $\delta=\frac{|1+\lambda r|}{1+\lambda r}.$
\end{corol}
In particular, if $A=\Id$ and $g=g_0$, this coincides with \cite[Corollary~5.7]{HKK2021}.

\begin{proof}
  By Lemma~\ref{lemm-F-S-norm}, there is a function $\kappa$ such that $\frac{1+\lambda r}{2!}D^2J_r(0)[x^2]=\kappa(x)x$. Then the left-hand term equals to $$\displaystyle \left\| \frac{1+r\lambda}{3!} D^3J_r(0)[x^3] -\mu \frac{1+r\lambda}{2!} \kappa(x) D^2J_r(0)[x^2]   \right\|. $$
  Lemma~\ref{lemm-F-S-norm} states that this is equal to 
  \begin{eqnarray*}
     && \left| \ell_x\left(  \frac{1+r\lambda}{3!} D^3J_r(0)[x^3] -\mu \frac{1+r\lambda}{2!} \kappa(x) D^2J_r(0)[x^2]      \right)   \right| \\
     &=& \left| a_3 -\mu a_2\kappa(x)    \right|  = \left| a_3 -\mu a_2^2   \right| .
     \end{eqnarray*}
Set $\mu=\nu-2+2\delta$. Then we proceed by Corollary~\ref{corol-h-assu2}:
     \begin{eqnarray*}
      &\le & \frac{r|q_1|}{|1+\lambda r|^3} \cdot \max \left( 1, \left|\frac{q_2}{q_1}- (2-\nu)\frac{r q_1}{|1+\lambda r|}\right|\right) \\
     &=& \frac{r|q_1|}{|1+\lambda r|^3}\cdot \max \left( 1, \left|\frac{q_2}{q_1}- (2\delta-\mu)\frac{r q_1}{|1+\lambda r|}\right|\right).
  \end{eqnarray*}
\end{proof}

\end{document}